\documentclass[12pt]{amsart}
\usepackage{amsmath}
\usepackage{amssymb,enumerate,multicol, commath}
\usepackage{hyperref}
\usepackage{cite}
 \usepackage{tikz}
  \usetikzlibrary{arrows,decorations.pathmorphing,backgrounds,positioning,fit}
  \usetikzlibrary{positioning}
  \tikzset{
    vertex/.style={circle,draw,minimum size=1.5em},
    edge/.style={->,> = latex'}
}

\newtheorem{theorem}{Theorem}[section]
\newtheorem*{theorem*}{Theorem}
\newtheorem{proposition}[theorem]{Proposition}
\newtheorem{lemma}[theorem]{Lemma}
\newtheorem{corollary}[theorem]{Corollary}
\theoremstyle{definition}
\newtheorem{definition}[theorem]{Definition}
\newtheorem{example}[theorem]{Example}

\DeclareMathOperator{\Lim}{Lim}

\DeclareMathOperator{\cl}{cl}

\begin{document}
\title[Dimension bounds for subfractals induced by $S$-gap shifts]{Hausdorff and box dimension bounds for subfractals induced by $S$-gap shifts}
\author{Elizabeth Sattler}
\address{Lawrence University, 711 E John St, Appleton, WI 54911}
\email{sattlere@lawrence.edu}

\maketitle

\begin{abstract}
In this paper, we consider subsets of an attractor of an iterated function system in which each point is associated with an allowable word from an $S$-gap shift.  The main result shows that bounds for the box dimension and Hausdorff dimension of a subfractal induced by an $S$-gap shift are given by the zeros of the upper and lower topological pressure functions. 
\end{abstract}
%%%%%%%%%%%
\section{Introduction}
%%%%%%%%%%%
Symbolic spaces are often used as a tool in fractal geometry.  For an iterated function system (IFS) of the form $\{K; f_1, f_2, \ldots f_p\}$, we consider the alphabet $\{1, 2, \ldots, p \}$ and the set of all one-sided infinite words from the alphabet.  Each point in the attractor of the IFS is associated with an infinite word.  In this paper, we will consider subfractals induced by a specific type of subshift; that is, we only consider points in an attractor that are associated with allowable words from a subshifts.\\

Fractal dimension computations for attractors of IFSs exist for several different types of IFSs: self-similar IFSs \cite{Hutchinson, FalconerFG}, non-self-similar attractors of hyperbolic IFSs \cite{EllisBranton}, and recurrent IFSs with hyperbolic maps \cite{Roychowdhury}. In the latter two on this list, the authors consider a subset of the full attractor of the IFS.  
Expanding on this idea, subfractals induced by subshifts are introduced in \cite{Sattler}.  In that article, computations for bounds on the fractal dimensions of a subfractal induced by a subshift are produced by modifying techniques used to compute the entropy of the underlying sofic subshift from \cite{LindMarcus}.  The proofs rely heavily on the existence of an adjacency matrix for the subshift, which is only guaranteed for a sofic subshift.  For non-sofic subshifts, there is no unified approach for entropy calculations.  \\

For an IFS $\{K; f_1, \ldots, f_p\}$ and a sofic subshift $Y$, let $\mathcal{F}_{Y}$ denote the attractor consisting only of points associated with allowable words in $Y$.  Assume for each $1 \leq i \leq p$ there exist constants $0< c_i \leq \bar{c}_i<1$ such that  $f_i$ satisfies $c_i d(x,y) \leq d(f_i(x),f_i(y)) \leq \bar{c}_i d(x,y)$. 
In both \cite{Roychowdhury} and \cite{Sattler}, a topological pressure function of the form $P(t) = \lim_{n \to \infty} \frac{1}{n} \log \left( \sum_{\omega \in \mathcal{L}_n} c_\omega^t \right)$ is used, where $\mathcal{L}_n$ denotes the allowable words of length $n$, and a function $\bar{P}(t)$ is defined similarly using $\bar{c}_i$.  
In \cite{Sattler}, it is shown that Hausdorff and box dimensions of $\mathcal{F}_{Y}$ are bounded by the zeros of $P(t)$ and $\bar{P}(t)$.  \\

Using a similar approach, we consider subfractals induced by $S$-gap shifts in this paper.  An $S$-gap shift is a subshift on the alphabet $\mathcal{A} = \{0,1\}$ defined by a set $S \subseteq \mathbb{N}_0$ where elements in $S$ describe the allowable number of 0s that can separate two 1s.  Although some $S$-gap shifts are sofic, we cannot guarantee the existence of a finite graph presentation to help us understand the structure of an $S$-gap shift.  Instead, we consider a decomposition of the language to establish bounds on the growth rate of finite words.  Techniques for decomposing the language of an $S$-gap shift and finding growth rate bounds are given in \cite{CT}.

\medskip

In the main theorem of this article, we prove that similar bounds for the Hausdorff and box dimension hold for subfractals induced by $S$-gap shifts:

\begin{theorem*}
Let $h$, $H$ be the unique values such that $P(h) = 0 = \bar{P}(H)$.  Then,
\begin{center} $h \leq \dim_H(\mathcal{F}_{X(S)}) \leq H$, 
$h \leq \underline{\dim}_B(\mathcal{F}_{X(S)}) \leq H$, and 
$h \leq \overline{\dim}_B(\mathcal{F}_{X(S)}) \leq H$.
\end{center}
\end{theorem*}

Although the assertions are similar, the proofs require a major departure from those in \cite{Roychowdhury, Sattler}.  The novelty in the proof is the introduction of an auxillary function, $Q(t)$, that enables us to incorporate structural properties of $S$-gap shifts in computations involving the topological pressure function.  In section 3, we reveal a clearer connection between the fractal dimension of such an attractor and the entropy of the underlying shift space.  In \cite{Spandl}, it is shown that the entropy of an $S$-gap shift is given by $\log(\lambda^{-1})$, where $\lambda$ satisfies
\[ \sum_{s \in S} \lambda^{s+1} = 1.\]
In this paper, we show that the values of $h$, $H$ for which $P(h) = 0 = \bar{P}({H})$ are the same values for which
\[ \sum_{s\in S} c_0^{sh} c_1^h =1 \text{  and  } \sum_{s \in S} \bar{c}_0^{sH} \bar{c}_1^H = 1.\]
Given a set $S$, this allows us to either compute or estimate $h$ and $H$, depending on the set $S$.   There is no uniform approach to entropy computations of non-sofic shift spaces; these bounds provide a technique for fractal dimension computations for subfractals induced by non-sofic $S$-gap shifts, such as the prime gap shift.
%%%%%%%%%%%%%%%%%%%
\section{Background and definitions}
%%%%%%%%%%%%%%%%%%%

Let $\mathcal{A} = \{0,1\}$ and $X = \{\omega_1 \omega_2 \omega_3 \ldots: \omega_i \in \mathcal{A} \text{ for all } i \geq 1\}$ be the compact metric space of all one-sided infinite words from letters in $\mathcal{A}$ equipped with the metric $d_X$ defined by $d_X(\omega, \tau) = \frac{1}{k}$, where $k = \min\{i : \omega_i \neq \tau_i\}$ . We consider the shift space $(X, \sigma)$, where the shift map $\sigma: X \to X$ is defined by $\sigma(\omega_1 \omega_2 \omega_3 \ldots ) = \omega_2 \omega_3 \omega_4 \ldots$. 

\medskip

A \textit{subshift} is a subset of $X$ that is both $\sigma$-invariant and closed.  There are several classes of well-studied subshifts, such as subshifts of finite type (SFTs) and sofic subshifts.  In this paper, we will focus on another class of subshifts called $S$-gap shifts.  An $S$-gap shift is defined using a subset $S \subset \mathbb{N}_0$.  $S$-gap shifts are usually defined as a subset of double-sided infinite words, where consecutive 1s are separated by $m$ 0s, where $m \in S$.  For our purpose of defining subfractals we will use one-sided infinite strings, and so an $S$-gap shift, denoted $X_S$, will be defined as the closure of the set 
\[ \{0^n 1 0^{s_1} 1 0^{s_2}1 \ldots : s_i \in S \text{ for all } i \geq 1\text{ and } n \in \mathbb{N}_0\}. \]
Here, we are assuming that $S$ is infinite; if $S$ is finite, we restrict the first block of 0s with $0 \leq n \leq \max\{s : s \in S\}$.  Sofic shifts include all SFTs; however, many $S$-gap shifts are not sofic.

\medskip

To better understand the structure of a subshift and the construction of a subfractal, we will focus on the allowable finite words in a subshift. Let $\mathcal{L}_n(X)$ denote the collection of all words from $\mathcal{A}$ of length $n$, and $\mathcal{L}(X) = \bigcup_{n=1}^\infty \mathcal{L}_n$. We call $\mathcal{L}(X)$ the \textit{language} of $X$.  Similarly, for a subshift $Y \subset X$, we will use $\mathcal{L}_n(Y)$ to denote the collection of finite words of length $n$ that appear in some infinite string in $Y$ and $\mathcal{L}(Y)$ will denote the language of $Y$.  If the subshift is understood in context, we will just use $\mathcal{L}_n$ and $\mathcal{L}$.  

\medskip

It will be helpful for us to consider special decompositions of a language.  We say a language $\mathcal{L}$ admits a decomposition of the form $\mathcal{L} = \mathcal{C}^p \mathcal{G} \mathcal{C}^s$ if every $\omega \in \mathcal{L}$ can be written as $\omega = \alpha_1 \gamma \alpha_2$ where $\alpha_1 \in \mathcal{C}^p$, $\gamma \in \mathcal{G}$, and $\alpha_2 \in \mathcal{C}^s$.  We can think of $\mathcal{C}^p$ and $\mathcal{C}^s$ as prefix and suffix sets respectively, and  $\mathcal{G}$ as a core set.  

\medskip

Let $\mathcal{K} \subset \mathbb{R}^n$ be a compact set and let $\{\mathcal{K}; f_0, f_1\}$ denote an iterated function system (IFS) with $f_i \colon \mathcal{K} \to \mathcal{K}$ for $i=0,1$.  We will assume that the IFS is \textit{hyperbolic}, meaning that there exist $0 < c_i \leq \bar{c_i} < 1$ such that
\[ c_i d(x,y) \leq d(f_i(x),f_i(y)) \leq \overline{c_i}d(x,y)\]
for all $x,y \in \mathcal{K}$ and for $i=0,1$.  Let $\mathcal{F}$ denote the attractor of this hyperbolic IFS.  Given this IFS, we consider the alphabet $\mathcal{A} = \{0,1\}$, where $0$ corresponds to $f_0$ and $1$ to $f_1$.  For a finite word $\omega  \in \mathcal{L}_n$, we adopt the following notations:
\begin{align*}
f_\omega &= f_{\omega_n} \circ f_{\omega_{n-1}} \circ \cdots f_{\omega_1} \\ 
c_\omega &= c_{\omega_1} c_{\omega_2} \cdots c_{\omega_n} \\
\bar{c}_\omega &= \bar{c} _{\omega_1} \bar{c}_{\omega_2} \cdots \bar{c}_{\omega_n} 
\end{align*}

Define the associated coding map $\pi \colon X \to \mathcal{F}$ by $\pi (\omega) = \lim_{n \to \infty} f_{\omega \vert_n} ( \mathcal{K})$, where $\omega \vert_n = \omega_1 \omega_2 \ldots \omega_n$.  For any IFS of the form $\{\mathcal{K} ; f_1 ,\ldots f_p\}$ and a subshift $Y \subset X = \mathcal{A}^\mathbb{N}$, we define a subfractal of $\mathcal{F}$ by only considering points associated with a word from $Y$.  More formally, we define $\mathcal{F}_{Y} = \{ \pi(\omega) : \omega \in Y\}$.  We will assume the IFS satisfies the \textit{open set condition (OSC)}, which means that there exists an open set $V \subset \mathcal{K}$ such that $\cup_{i=1}^p f_i(V) \subset V$, where the union is disjoint.  \\

\medskip

The main result of this paper involves Hausdorff and box dimensions, which we define here.  Let $\mathcal{K} \subset \mathbb{R}^n$ be a compact set, $E \subset \mathcal{K}$, and $\mathcal{O}$ be the collection of all open $\varepsilon$-covers of $E$. 
 For $s \geq 0$, if $\overline{\mathcal{H}}_\varepsilon^s(E) = \inf_{\mathcal{U} \in \mathcal{O}} \sum_{U \in \mathcal{U}} (\text{diam}(U))^s$ then the $s$-dimensional Hausdorff outer measure is defined by $\mathcal{H}^s(E)   = \lim_{\varepsilon \to 0} \overline{\mathcal{H}}^s_\varepsilon (E)$.  
 We restrict the outer measure to $\overline{\mathcal{H}}^s$-measurable sets to define the $s$-dimensional Hausdorff measure, denoted $\mathcal{H}^s$.  
 The \textit{Hausdorff dimension} of $E$, denoted $\dim_H(E)$ is the unique value of $s$ such that 
\[ \mathcal{H}^r(E) = \begin{cases} 0, &r > s \\
\infty, &r<s
\end{cases}
\]

\medskip

Let $N_r(E)$ be the smallest number of sets of diameter $r$ that can cover $E$.  The \textit{lower and upper box dimensions} of $E$ are defined, respectively, as:
\[ \underline{\dim}_B(E) = \liminf_{r \to 0} \frac{ \log N_r(E)}{-\log r} \text{  and  } \overline{\dim}_B(E) = \limsup_{r \to 0} \frac{\log N_r(E)}{-\log r}\]

The following relationship between the these fractal dimensions are well-known \cite{FalconerFG}:
\[ \dim_H(E) \leq \underline{\dim}_B(E) \leq \overline{\dim}_B(E)\]

%%%%%%%%%%%%%%%%%%%%%%%%%%%%
\section{Topological pressure function for subshifts}
%%%%%%%%%%%%%%%%%%%%%%%%%%

Topological pressure functions are sometimes considered to be generalizations of topological entropy.  For subshifts, the topological entropy is given by 
\[h(Y) = \lim_{n \to \infty} \frac{1}{n} \log ( \# \mathcal{L}_n(Y)). \]
Computations for the topological entropy of SFTs, sofic subshifts, and $S$-gap shifts are known \cite{LindMarcus,Spandl}.  Following \cite{Roychowdhury, Sattler} we define the following topological pressure function.

\begin{definition}
Let $Y$ be a subshift.  The \textit{lower topological pressure function} of $\mathcal{F}_{Y}$ is given by $P(t) = \lim_{n \to \infty} \frac{1}{n} \log \left( \sum_{\omega \in \mathcal{L}_n(Y)} c_\omega^t\right)$.  Similarly, the \textit{upper topological pressure} is defined by $\bar{P}(t) = \lim_{n \to \infty} \frac{1}{n} \log \left( \sum_{\omega \in \mathcal{L}_n(Y)} \bar{c}_\omega^t\right)$.

\end{definition}

In \cite{Sattler}, it is shown that both topological pressure functions are strictly decreasing, convex, and continuous on $\mathbb{R}$, and there are unique values $h,H \in [0, \infty)$ such that $P(h) = 0 = \bar{P}(H)$ where $h \leq H$.  Determining the zeros of a topological pressure function is a difficult process in general; it will be helpful for us to consider other closely related functions that will make it easier for us to find values of $h$ and $H$.  

\medskip

To define these new functions, we will utilize a suitable decomposition of the language of an $S$-gap shift, which is used in \cite{CT}.  We define a core set, prefix set, and suffix set as follows:
\begin{align*}
\mathcal{G} &= \{0^{s_1}10^{s_2}1\ldots0^{s_k}1: s_i \in S \text{ for } 1 \leq i \leq k\} \\
\mathcal{C}^p &= \{0^n1 : n \notin S\} \\
\mathcal{C}^s &= \{0^n : n \in \mathbb{N}_0\}
\end{align*}
Given an $S$-gap shift $X(S)$, notice that every word in $\mathcal{L}(X(S))$ can be written in the form $0^m \omega 0^n$ where $m,n \in \mathbb{N}_0$ and $\omega \in \mathcal{G}$, and so we can write $\mathcal{L}(X(S)) = \mathcal{C}^p \mathcal{G} \mathcal{C}^s$.  Similar to the notation for languages, we will use $\mathcal{G}_n$, $\mathcal{C}^p_n$ and $\mathcal{C}^s_n$ to denote the subsets consisting of words of length $n$.  Note that we have assumed $S$ is infinite; if $S$ is finite, then we restrict $n$ in both $\mathcal{C}^p$ and $\mathcal{C}^s$ with $n \leq \max (S)$.

\medskip

Using the core set $\mathcal{G}$ defined above, we define two functions:
\[ 
Q(t) = \limsup_{n \to \infty} \frac{1}{n} \log \left( \sum_{\omega \in \mathcal{G}_n} c_\omega^t\right) \text{  and  }
\bar{Q}(t) = \limsup_{n \to \infty} \frac{1}{n} \log \left( \sum_{\omega \in \mathcal{G}_n} \bar{c}_\omega^t\right)
\]
Our goal is to show that $P(t) = Q(t)$ and $\bar{P}(t) = \bar{Q}(t)$ for $t$ values near the zeros of these functions.  We begin by establishing bounds on $\sum_{\omega \in \mathcal{L}_n} c_\omega^t$ and $\sum_{\omega \in \mathcal{G}_n} c_\omega^t$.  The techniques are modified from the ideas used to establish bounds on the growth rates of $\mathcal{L}_n$ and $\mathcal{G}_n$ in \cite{CT}.  All proofs will be presented using $P(t)$ and $Q(t)$ with constants $c_i$, and proofs for the same results using $\bar{Q}(t)$ and $\bar{P}(t)$ with constants $\bar{c}_i$ are identical.

\begin{lemma}
\label{coreupperlemma} 
For any $n \in \mathbb{N}$,
\[\sum_{\omega \in \mathcal{G}_n} c_\omega^t\leq e^{n P(t)}. \]
\end{lemma}

\begin{proof}
First, we notice that $(\mathcal{G}_n)^k \subset \mathcal{L}_{nk}$.  Using this and some algebra, we obtain
\newline $\frac{1}{nk} \log \left(\sum_{\omega \in \mathcal{G}_n} c_\omega^t \right)^k \leq \frac{1}{nk} \log \left( \sum_{\omega \in \mathcal{L}_{nk}} c_\omega^t \right)$.  This implies that
\[\frac{1}{n} \log \left( \sum_{\omega \in \mathcal{G}_n} c_\omega^t \right)  \leq \lim_{k \to \infty} \frac{1}{nk} \log \left( \sum_{\omega \in \mathcal{L}_{nk}} c_\omega^t \right) = P(t),\]
and so $\sum_{\omega \in \mathcal{G}_n} c_\omega^t \leq e^{n P(t)}$.
\end{proof}

 \begin{lemma}
 \label{langboundlemma}
 For any $n \in \mathbb{N}$ and $t$ in a neighborhood of $h$,
 \[ e^{n P(t) } \leq \sum_{\omega \in \mathcal{L}_n} c_\omega^t \leq K_1 e^{n P(t)},\]
 where $K_1$ is a constant.
 \end{lemma} 

\begin{proof}
First, using the fact that $\mathcal{L}_{nk} \subset (\mathcal{L}_n)^k$, we obtain
\[ \frac{1}{nk} \log \left ( \sum_{\omega \in \mathcal{L}_{nk}} c_\omega^t \right) \leq \frac{1}{nk} \log \left( \sum_{\omega \in \mathcal{L}_n} c_\omega^t \right)^k = \frac{1}{n} \log \left ( \sum_{\omega \in \mathcal{L}_n} c_\omega^t\right).\]
Letting $k \to \infty$, we obtain $P(t) \leq \frac{1}{n}  \log \left ( \sum_{\omega \in \mathcal{L}_n} c_\omega^t\right)$, and so
$e^{nP(t)} \leq \sum_{\omega \in \mathcal{L}_n} c_\omega^t$.

\medskip

To find an upper bound, we will work with the language decomposition of $\mathcal{L}$.  Recall that $\mathcal{C}_n^p = \{0^{n-1}1\}$ or $\emptyset$ and $\mathcal{C}_n^s = \{0^n\}$.  Therefore, $\sum_{\omega \in \mathcal{C}^p_n} c_\omega^t = c_0^{(n-1)t}c_1^{t}$ when $n-1 \notin S$ and $\sum_{\omega \in \mathcal{C}^s_n} c_\omega^t = c_0^{nt}$.  \\

Next, recall that $P(t)$ is a continuous function and $P(h) = 0$, so there exists a neighborhood $N_\alpha(h)$ such that  $\log(c_0^t) < P(t)$ and $\log(c_1^t) < P(t)$ for $t \in N_\alpha(h)$.  Therefore, we can find $\varepsilon>0$ such that both $\log(c_0^t) \leq P(t) - \varepsilon$ and $\log(c_1^t) \leq P(t) - \varepsilon$.  This implies that $c_0^{tn} \leq e^{n(P(t)-\varepsilon)}$ and $c_1^{t} \leq e^{P(t)-\varepsilon}$.

\medskip

Now, we use this fact with Lemma \ref{coreupperlemma} to obtain 
\begin{align*}
\sum_{\omega \in \mathcal{L}_n} c_\omega^t &\leq \sum_{i+j+k=n} \left( \sum_{\omega \in \mathcal{C}_i^p} c_\omega^t \right)  \left( \sum_{ \tau \in \mathcal{G}_j} c_\tau^t \right) \left( \sum_{\gamma \in \mathcal{C}_k^s} c_\gamma^t \right) \\
&\leq \sum_{i+j+k=n} c_0^{(i+k-1)t}c_1^t e^{jP(t)} \\
&\leq \sum_{i+j+k=n} e^{(i+k)(P(t)- \varepsilon)} e^{jP(t)} \\
&= e^{n P(t)} \sum_{i+j+k=n} e^{-\varepsilon(i+k)} \\
&= e^{nP(t)} \sum_{m=0}^n \sum_{i=0}^m e^{-m \varepsilon}\\
& \leq e^{nP(t)} \sum_{m=0}^\infty (m+1) e^{-m \varepsilon}
\end{align*}
Since $\sum_{m=0}^\infty (m+1) e^{-m \varepsilon}$ converges, we can find $K_1$ such that $\sum_{\omega \in \mathcal{L}_n} c_\omega^t \leq K_1 e^{nP(t)}$.
\end{proof}

\begin{corollary}
\label{corboundh}
Let $h$ be the unique value such that $P(h) = 0$.  Then, there exists constant $K_2$ such that 
\[1 \leq \sum_{\omega \in \mathcal{L}_n} c_\omega^h \leq K_2.\]

\end{corollary}

\begin{corollary}
\label{corboundH}
Let $H$ be the unique value such that $\bar{P}(H) = 0$.  Then, there exists constant $L_2$ such that 
\[1 \leq \sum_{\omega \in \mathcal{L}_n} \bar{c}_\omega^H \leq L_2.\]

\end{corollary}

 \begin{lemma}
 \label{corelowerbound}
 There exist constants $K_3>0$ and $N \in \mathbb{N}$ such that for every $n \in \mathbb{N}$, there exists $l$ satisfying $n-N \leq l \leq n$ such that 
 \[ K_3e^{lP(t)} \leq \sum_{\omega \in \mathcal{G}_l} c_\omega^t\]
 for $t$ in a neighborhood of $h$.
 \end{lemma}

\begin{proof}
By Lemma \ref{langboundlemma}, there exists a constant $K_1$ with $\sum_{\omega \in \mathcal{L}_n} c_\omega^t \leq K_1 e^{n P(t)}$.  For $j \in \mathbb{N}$, let $a_j = e^{-jP(t)} \sum_{\tau \in \mathcal{G}_j} c_\tau^t$.  Choose $\varepsilon>0$ as in the proof of Lemma \ref{langboundlemma} and following a similar approach, we find that 
\begin{align*}
e^{nP(t)} \leq \sum_{\omega \in \mathcal{L}_n} c_\omega^t &\leq \sum_{i+j+k=n} \left( \sum_{ \omega \in \mathcal{C}^p_i} c_\omega^t \right) \left(  \sum_{\tau \in \mathcal{G}_j} c_\tau^t \right) \left( \sum_{\gamma \in \mathcal{C}^s_k} c_\gamma^t  \right) \\
&\leq \sum_{i+j+k=n} e^{(i+k)(P(t) - \varepsilon)} \sum_{\tau \in \mathcal{G}_j} c_\tau^t
\end{align*}
By dividing both sides by $e^{nP(t)}$, we get 
\begin{align*}
1 &\leq \sum_{i+j+k=n} e^{-jp(t)} e^{(i+k)(-\varepsilon)} \sum_{\tau \in \mathcal{G}_j} c_\tau^t \\
&= \sum_{i+j+k=n} e^{(i+k)(-\varepsilon)}a_j \\
&\leq \sum_{m=0}^n (m+1)e^{-m\varepsilon}a_{n-m}\\
\end{align*}
Choose $N$ large enough such that $1-\sum_{m=N}^\infty (m+1)e^{-m\varepsilon} > 0$, and let $M = \max \{(m+1)e^{-m\varepsilon} : m \in \mathbb{N}\}$.  Using the inequalities above and Lemma \ref{coreupperlemma}, we get
\[ 1 \leq \sum_{m=0}^{N-1} Ma_{n-m} + \sum_{m=N}^\infty (m+1)e^{-m\varepsilon}.\]
By rearranging, we obtain:
\[ M \sum_{m=0}^{N-1} a_{n-m} \geq 1- \sum_{m=N}^\infty (m+1)e^{-m \varepsilon}>0.\]
This implies that there exists an $l$ with $n-N \leq l \leq n$ satisfying \newline $N M a_l \geq 1-\sum_{m=M}^\infty (m+1)e^{-m\varepsilon}$, and so
\[
\sum_{\omega \in \mathcal{G}_l} c_\omega^t \geq \frac{1-\sum_{m=N}^\infty e^{-m\varepsilon}}{MN} e^{lP(t)}.
\]
\end{proof} 

\begin{proposition}
For $t$ in a neighborhood of $h$, $Q(t) = P(t)$.
\end{proposition}

\begin{proof}
For any $n \in \mathbb{N}$, $\mathcal{G}_n \subset \mathcal{L}_n$, and so $\sum_{\omega \in \mathcal{G}_n} c_\omega^t \leq \sum_{\omega \in \mathcal{L}_n} c_\omega^t$.  It follows that $Q(t) \leq P(t)$.\\

\medskip

Next, using Lemma \ref{corelowerbound}  we find constants $K_3$ and $N$ to construct  a sequence $(l_n)$ for which $n-N \leq l_n \leq n$ and $K_3 e^{l_n P(t)} \leq \sum_{\omega \in \mathcal{G}_{l_n}} c_\omega^t$.  Using this with Lemma \ref {langboundlemma}, we obtain the inequality
\[ \frac{K_3}{K_1 e^{N P(t)}} \sum_{\omega \in \mathcal{L}_n} c_\omega^t \leq K_3 e^{(n-N)P(t)} \leq K_3 e^{l_n P(t)} \leq \sum_{\omega \in \mathcal{G}_{l_n}} c_\omega^t.
\]
Building off of this inequality, we obtain
\begin{align*}
P(t)&= \lim_{n \to \infty} \frac{1}{n} \log \left( \frac{K_3}{K_1e^{N P(t)}} \sum_{\omega \in \mathcal{L}_n} c_\omega^t \right) \\
&\leq \lim_{n \to \infty} \frac{1}{n} \log ( \sum_{\omega \in \mathcal{G}_{l_n}} c_\omega^t) \\
&\leq \lim_{n \to \infty} \frac{1}{l_n} \log ( \sum_{\omega \in \mathcal{G}_{l_n}} c_\omega^t)\\
& \leq \limsup_{n \to \infty} \frac{1}{n} \log ( \sum_{\omega \in \mathcal{G}_n} c_\omega^t) = Q(t).
\end{align*}
Therefore, $Q(t) = P(t)$.
\end{proof}

An analogous argument proves the following.

\begin{proposition}
For $t$ in a neighborhood of $h$, $\bar{Q}(t) = \bar{P}(t)$.
\end{proposition}

 In the next section, we will show that the values $h$ and $H$ for which $P(h) = 0$ and $\bar{P}(H)=0$ give bounds for the box and Hausdorff dimension of the attractor, and so we are interested in $Q(h) =0$ and $\bar{Q}(H) = 0$. This leads us to our next proposition that will provide a better computational tool for finding these special values for $h$ and $H$. 

\begin{proposition}
\label{compprop}
For real number $h>0$, $Q(h) = 0$ if and only if $\sum_{s\in S} c_0^{sh}c_1^h=1$.
\end{proposition}

\begin{proof}
First, we let $Q(h)=0$ and consider the series $\sum_{n \geq 1} \sum_{\omega \in \mathcal{G}_n} c_\omega^t$.  
If $t<h$, then $\limsup_{n \to \infty} ( \sum_{\omega \in \mathcal{G}_n} c_\omega^t)^{1/n} = e^{Q(t)} > 1$, so by the root test, $\sum_{n \geq 1} \sum_{\omega \in \mathcal{G}_n} c_\omega^t$ diverges.  
Similarly, if $t>h$, then $\sum_{n \geq } \sum_{\omega \in \mathcal{G}_n} c_\omega^t$ converges. 
Next, notice that 
\begin{align}
\sum_{n \geq 1} \sum_{\omega \in \mathcal{G}_n} c_\omega^t = \sum_{k \geq 1} ( \sum_{s \in S} c_0^{st} c_1^t)^k,
\end{align}
and $\sum_{k \geq 1} ( \sum_{s \in S} c_0^{st} c_1^t)^k$ converges when $\sum_{s \in S} c_0^{st}c_1^t < 1$ and diverges when $\sum_{s \in S} c_0^{st} c_1^t >1$.  
Therefore, $\sum_{s \in S} c_0^{sh} c_0^h = 1$ where $h$ satisfies $Q(h)=0$.\\

\medskip

Next, assume that $\sum_{s \in S} c_0^{sh}c_1^h=1$.  Notice that $\sum_{s \in S} c_0^{st} c_1^t$ is decreasing with respect to $t$ and so $\sum_{k \geq 1} (\sum_{s \in S} c_0^{st} c_1^t)^k$ converges for $t>h$ and diverges for $t<h$.  
Therefore, $\sum_{n \geq 1} \sum_{\omega \in \mathcal{G}_n} c_\omega^t$ converges for $t>h$ and diverges for $t<h$.  
Because $e^{Q(t)} = \limsup_{n \to \infty} ( \sum_{\omega \in \mathcal{G}_n} c_\omega^t)^{1/n}$ is strictly decreasing, we can consider all possible outcomes for the root test applied to $\sum_{n \geq 1} \sum_{\omega \in \mathcal{G}_n} c_\omega^t$ and find that it must be the case the $e^{Q(h)} =1$ and so $Q(h)=0$.
\end{proof}

Again, an analogous argument proves the following. 

\begin{proposition} For real number $H>0$,
$\bar{Q}(H) = 0$ if and only if $\sum_{s\in S} \bar{c}_0^{sH}\bar{c}_1^H=1$.
\end{proposition}

%%%%%%%%%%%%%%%%%%
\section{Main Theorem} 
%%%%%%%%%%%%%%%%%%%

To establish the lower bound for the Hausdorff dimension of a subfractal induced by an $S$-gap shift, we must find a suitable Borel probability measure that is supported on the subfractal that will satisfy the Mass Distribution Principle \cite{FalconerTech}.
  For that purpose, we define a measure on the symbolic space.  For $\omega \in \mathcal{L}(X)$, let $[ \! [ \omega ]\!] = \{ \tau \in X : \tau_i = \omega_i \text{ for } 1 \leq i \leq \ell(\omega)\}$ denote the cylinder set of $\omega$.  We fix $h$ to be the unique value such that $P(h) = 0$, let $\omega \in \mathcal{L}(X(S))$ and define 
\[ \nu_n( [ \! [ \omega]\! ] ) = \frac{\sum_{\omega \tau \in \mathcal{L}_{n + \ell(\omega)}} c_{\omega \tau}^h} {\sum_{\tau \in \mathcal{L}_{n+\ell(\omega)}} c_\tau^h}.\] 
Next, notice that for $n \geq 1$ and any $\omega \in \mathcal{L}$, by Corollary \ref{corboundh} we have
\[ 0 \leq \frac{\sum_{\omega \tau \in \mathcal{L}_{n+\ell(\omega)}} c_{\omega \tau}^h }{K_2} 
\leq \nu_n( [\![ \omega ]\!]) 
\leq \frac{c_\omega^h \sum_{\tau \in \mathcal{L}_n} c_\tau^h}{\sum_{\tau \in \mathcal{L}_{n+\ell(\omega)}} c_\tau^h} 
\leq K_2 c_\omega^h < \infty.\]

Therefore, we can use the Banach limit to define $\nu( [\![ \omega ]\!] ) = \Lim_{n \to \infty} \nu_n ([\![ \omega ]\!])$.  Next, notice that :
\begin{align*}
 \sum_{i=1}^p \nu ( [\![ \omega i ]\!]) &= \Lim_{n \to \infty} \sum_{i=1}^p \frac{\sum_{\omega i \tau \in \mathcal{L}_{n + \ell(\omega i)}} c_{\omega i \tau}^h }{\sum_{\tau \in \mathcal{L}_{n+ \ell(\omega i)}} c_\tau^h} \\
&= \Lim_{n \to \infty} \frac{\sum_{\omega \tau \in \mathcal{L}_{n+1+\ell(\omega)}} c_{\omega \tau}^h}{\sum_{\tau \in \mathcal{L}_{n + 1 + \ell(\omega)}} c_\tau^h} = \nu( [\![ \omega ]\!])
\end{align*}

By the Kolmogorov Extension Theorem, we can extend $\nu$ to a unique Borel probability measure $\xi$ on all of $X(S)$.  We can use this measure to define a measure $\mu_h = \xi \circ \pi^{-1}$, where $\pi$ is the coding map $\pi: X(S) \to \mathbb{R}^n$.  The measure $\mu_h$ is supported on $\mathcal{F}_{X(S)}$.  

\medskip

For the proof of our main theorem, we will use a special cover of the subfractal $\mathcal{F}_{X(S)}$.  For $0<r<1$, we define the cover $\mathcal{U}_r = \{ f_\omega(\mathcal{K}) : \abs{f_\omega(\mathcal{K})} < r \leq \abs{f_{\omega^-}(\mathcal{K})} \}$, where $\omega^- = \omega_1 \omega_2 \ldots \omega_{n-1}$ if $\omega \in \mathcal{L}_n$.  From \cite{Sattler}, we know that for $x \in \mathcal{F}_{X(S)}$, the open ball $B(x,r)$ intersects at most $M$ elements of $\mathcal{U}_r$, where $M$ is independent of $r$.  This gives us the last tool we need to complete the proof of the main theorem.

\begin{theorem}
\label{Sgaphausbound}
Let $h$ and $H$ be the unique values such that $P(h) = 0 = \bar{P}(H)$.  Then, $h \leq \dim_H(\mathcal{F}_{X(S)}) \leq H$, $h \leq \underline{\dim_B}(\mathcal{F}_{X(S)}) \leq H$, and  $h \leq \overline{\dim_B}(\mathcal{F}_{X(S)}) \leq H.$
\end{theorem}

\begin{proof}
We will utilize the following well-known relationships between Hausdorff and box dimension \cite{FalconerFG}: 
\[ \dim_H(\mathcal{F}_{X(S)}) \leq \underline{\dim}_B(\mathcal{F}_{X(S)}) \leq \overline{\dim}_B(\mathcal{F}_{X(S)}).\]

%For the upper bound, we start by finding a suitable cover that will show us the $H$-dimensional Hausdorff measure is finite.  Consider the cover $\mathcal{U}_n = \{ f_\omega(\mathcal{K}) : \omega \in \mathcal{L}_n\}$.  Notice that $\mathcal{U}_n$ is a cover for all $n \geq 1$ and $\diam(U) \to 0$ as $n \to \infty$ for any $U \in \mathcal{U}_n$.  Let $\mathcal{E}_\delta$ denote the collection of all $\delta$-covers of $\mathcal{F}_{X(S)}$.  By  Corollary \ref{corboundH}, we have
%\begin{align*}
%\mathcal{H}^H(\mathcal{F}_{X(S)}) &= \lim_{\delta \to 0} \inf_{\mathcal{E} \in \mathcal{E}_\delta} \sum_{E \in \mathcal{E}} \abs{E}^H 
%\leq \lim_{n \to \infty} \sum_{\omega \in \mathcal{L}_n} \abs{f_\omega(\mathcal{K})}^H \\
%&\leq \lim_{n \to \infty} \sum_{\omega \in \mathcal{L}_n} \abs{\mathcal{K}}^H \bar{c}_\omega^H 
%\leq \abs{\mathcal{K}}^H L_2 < \infty
%\end{align*}
%Therefore, $\dim_H(\mathcal{F}_{X(S)}) \leq H$.

We start by proving $h$ is a lower bound for the Hausdorff dimension.
Next, let $x \in \mathcal{F}_{X(S)}$ and consider the open ball $B(x,r)$ for some $0<r<1$ and the cover $\mathcal{U}_r = \{f_\omega(\mathcal{K}) : \abs{f_\omega(\mathcal{K})} < r \leq \abs{f_{\omega^-}(\mathcal{K})} \}$.  We know at most $M$ elements of $\mathcal{U}_r$ intersect $B(x,r)$, so let $\mathcal{U}_M = \{ f_\omega(\mathcal{K}) \in \mathcal{U}_r : f_\omega(\mathcal{K}) \cap B(x,r) \neq \emptyset\}$ and let $\mathcal{L}_M = \{\omega : f_\omega(\mathcal{K}) \in \mathcal{U}_M\}$.  Notice that for any $f_\omega(\mathcal{K}) \in \mathcal{U}_M \subseteq \mathcal{U}_r$ and corresponding $\omega \in \mathcal{L}_M$,

\begin{align}
c_\omega \abs{\mathcal{K}} &\leq \abs{f_\omega(\mathcal{K})} < r \text{  and  }\\
\nu ( [\![ \omega ]\!] ) &= \Lim_{n \to \infty} \frac{\sum_{\omega \tau \in \mathcal{L}_{n+ \ell(\omega)}} c_{\omega \tau}^h}{\sum_{\tau \in \mathcal{L}_{n+\ell(\omega)}} c_\tau^h} \leq \Lim_{n \to \infty} \frac{c_\omega^h \sum_{\tau \in \mathcal{L}_{n}} c_{\tau}^h}{\sum_{\tau \in \mathcal{L}_{n+\ell(\omega)}} c_\tau^h} \leq K_2 c_\omega^h.
\end{align}
Using (2) and (3), we find that
\begin{align*}
\frac{\mu_h(B(x,r))}{r^h} &\leq \frac{\sum_{\omega \in \mathcal{L}_M} \mu_h(f_\omega(\mathcal{K}))}{r^h} 
= \frac{\sum_{\omega \in \mathcal{L}_M} \nu([\![ \omega ]\!])}{r^h} \\
&\leq  \frac{\sum_{\omega \in \mathcal{L}_M} K_2 c_\omega^h}{r^h}  \leq \frac{M K_2 r^h}{r^h \abs{\mathcal{K}^h}} \\
&= \frac{MK_2}{\abs{\mathcal{K}}^h}
\end{align*}
This inequality holds for all $0<r<1$, and so $\limsup_{r \to 0} \frac{\mu_h(B(x,r))}{r^h} \leq \frac{MK_2}{\abs{\mathcal{K}}^h}$.  By the Mass Distribution Principle, it follows that $\mathcal{H}^h(\mathcal{F}_{X(S)}) \geq \frac{\mu_h(\mathcal{F}_{X(S)})\abs{\mathcal{K}}^h}{M K_2} > 0$.  Hence, $ h \leq \dim_H (\mathcal{F}_{X(S)})$ and so $h$ is also a lower bound for the box dimensions of $\mathcal{F}_{X(S)}$.\\

\medskip
It remains to show that $\overline{\dim}_B(\mathcal{F}_{X(S)}) \leq H$.  To construct covers with minimal overlap, we start with an open set $V$ as in the definition of the OSC and consider its closure $\cl (V)$.  Let $0 < \delta <1$ and consider the cover $\mathcal{U}_\delta = \{f_\omega(\cl (V)) : \abs{f_\omega(\cl (V))} < \delta \leq \abs{f_{\omega^-}(\cl (V))} \}$ and the collection of associated finite words $\mathcal{L}_\delta = \{ \omega : f_\omega (\cl (V)) \in \mathcal{U}_\delta\}$.  
For $k = \min \{ \ell(\omega) : \omega \in \mathcal{L}_\delta\}$, we define another cover $\mathcal{U}_k = \{ f_\omega(\cl (V)) : \omega \in \mathcal{L}_k\}$.  Notice that $\cup_{U \in \mathcal{U}_\delta} U \subseteq \cup_{U \in \mathcal{U}_k} U$.  Using this fact with Corollary \ref{corboundH}, it follows that 
\begin{align}
 \sum_{\omega \in \mathcal{L}_\delta} \abs{f_\omega(\cl (V))}^H 
\leq \sum_{\omega \in \mathcal{L}_k} \abs{f_\omega(\cl (V))}^H
\leq \sum_{\omega \in \mathcal{L}_k} \bar{c}_\omega^H \abs{\cl (V)}^H 
\leq L_2 \abs{\cl(V)}^H
\end{align}
The cover $\mathcal{U}_\delta$ contains finitely many elements, so there must exist some $\tau \in \mathcal{L}_\delta$ associated with the element of smallest diameter.  Specifically, let $\tau$ be the word such that  $\abs{f_\tau(\cl (V))} \leq \abs{f_\omega(\cl (V))}$ for all $\omega \in \mathcal{L}_\delta$.  For this $\tau$ and using the definition of $\mathcal{U}_\delta$, we obtain
\begin{align}
\sum_{\omega \in \mathcal{L}_\delta} \abs{f_\omega(\cl (V))}^H 
\geq \abs{f_\tau(\cl(V))}^H \abs{\mathcal{L}_\delta} 
\geq \abs{f_{\tau^-}(\cl (V))}^H c_{min}^H \abs{\mathcal{L}_\delta}
\geq \delta^H c_{min}^H \abs{\mathcal{L}_\delta}.
\end{align}

Let $N_\delta(\mathcal{F}_{X(S)})$ denote the smallest number of sets of diameter at most $\delta$ that cover $\mathcal{F}_{X(S)}$.  Using (4) and (5), we obtain
\[ L_2 \abs{\cl (V)}^H \geq \delta^H c_{min}^H \abs{\mathcal{L}_\delta} \geq \delta^H c_{min}^H N_\delta(\mathcal{F}_{X(S)}).\]

Therefore, $N_\delta (\mathcal{F}_{X(S)}) \leq \frac{L_2 \abs{\cl (V)}^H}{\delta^H c_{min}^H}$.  Using this we find that 
\[
\frac{\log (N_\delta(\mathcal{F}_{X(S)}))}{-\log \delta} \leq \frac{\log(L_2\abs{\cl (V)}^H) - \log(c_{min}^H) - H \log \delta}{-\log \delta}
\]
Therefore, by letting $\delta \to 0$, we find that 
\[
\overline{\dim}_B(\mathcal{F}_{X(S)}) = \limsup_{\delta \to 0} \frac{\log(N_\delta(\mathcal{F}_{X(S)}))}{-\log \delta} \leq H.
\]
It follows that $H$ is also an upper bound for the Hausdorff dimension and lower box dimension.
\end{proof}

\begin{example}
Let $X_F$ denote the Golden Mean shift,  a subshift of finite type on $\mathcal{A} = \{0,1\}$ with forbidden word list $F = \{11\}$.  We can also view $X_F$ as an $S$-gap shift with set $S = \mathbb{N}$.  Let $\mathcal{F}_{X_F}$ denote the subfractal induced by $X_F$, where $\mathcal{F}$ is an attractor of an IFS with maps $f_0$, $f_1$ and these maps have lower and upper contractive bounds of $c_0$, $\bar{c_0}$, $c_1$, $\bar{c_1}$, respectively.  Using the results above, we see that both the Hausdorff and box dimension of $\mathcal{F}_{X_F}$ is bounded by $h$ and $H$, where $P(h) = \bar{P}(H) = 0$.  Using Proposition \ref{compprop}, we find that 
\[ 1 = \sum_{s \in S} c_0^{sh}c_1^h = \sum_{n=1}^\infty c_0^{nh}c_1^h = \frac{c_1^hc_0^h}{1-c_0^h},\]
and a similar result involving $H$.  Therefore, we are interested in the values of $h$ and $H$ for which $c_0^h+c_0^hc_1^h=1$ and $\bar{c}_0^H+\bar{c}_0^H \bar{c}_1^H=1.$\\
We can also compute these bounds using techniques from \cite{Sattler} by viewing $X_F$ as an SFT.  In this approach, we must compute the maximal eigenvalue of 
$\begin{bmatrix}
c_0^t & c_1^t \\
c_0^t & 0 
\end{bmatrix}$
which is $\rho(t) = \frac12 ( c_0^t + \sqrt{c_0^{2t} + 4 c_0^tc_1^t})$.  We are interested in the value $h$ for which $\rho(h) = 1$, and with some algebra we find that $c_0^h + c_0^h c_1^h=1$.  Through a similar process, we also find $H$ by $\bar{c}_0^H + \bar{c}_0^H \bar{c}_1^H=1$.
\end{example}

\section{Acknowledgments}
The author thanks Do\u{g}an \c{C}\"{o}mez and Scott Corry for their helpful feedback and suggestions on this work.

\bibliographystyle{plain}
\bibliography{sgapbib}

\begin{thebibliography}{1}

\bibitem{CT}
Vaughn Climenhaga and Daniel Thompson.
\newblock Intrinsic ergodicity beyond specification: $\beta$-shifts, $s$-gap
  shifts, and their factors.
\newblock {\em Israel J. Math.}, 192(2):785–817, 2012.

\bibitem{EllisBranton}
David~B. Ellis and Michael~G. Branton.
\newblock Non-self-similar attractors of hyperbolic iterated function systems.
\newblock In {\em Dynamical systems ({C}ollege {P}ark, {MD}, 1986--87)}, volume
  1342 of {\em Lecture Notes in Math.}, pages 158--171. Springer, Berlin, 1988.

\bibitem{FalconerTech}
Kenneth Falconer.
\newblock {\em Techniques in fractal geometry}.
\newblock John Wiley \& Sons, Ltd., Chichester, 1997.

\bibitem{FalconerFG}
Kenneth Falconer.
\newblock {\em Fractal geometry}.
\newblock John Wiley \& Sons, Ltd., Chichester, third edition, 2014.
\newblock Mathematical foundations and applications.

\bibitem{Hutchinson}
John~E. Hutchinson.
\newblock Fractals and self-similarity.
\newblock {\em Indiana Univ. Math. J.}, 30(5):713--747, 1981.

\bibitem{LindMarcus}
Douglas Lind and Brian Marcus.
\newblock {\em An introduction to symbolic dynamics and coding}.
\newblock Cambridge University Press, Cambridge, 1995.

\bibitem{Roychowdhury}
Mrinal~Kanti Roychowdhury.
\newblock Hausdorff and upper box dimension estimate of hyperbolic recurrent
  sets.
\newblock {\em Israel J. Math.}, 201(2):507--523, 2014.

\bibitem{Sattler}
Elizabeth Sattler.
\newblock Fractal dimensions of subfractals induced by sofic subshifts.
\newblock {\em Monatsh. Math.}, 183(3):539--557, 2017.

\bibitem{Spandl}
Christoph Spandl.
\newblock Computing the topological entropy of shifts.
\newblock {\em MLQ Math. Log. Q.}, 53(4-5):493--510, 2007.

\end{thebibliography}

\end{document}